\documentclass[12pt]{amsart}

\usepackage{amssymb}

\usepackage{enumerate}

\makeatletter
\@namedef{subjclassname@2010}{%
  \textup{2010} Mathematics Subject Classification}
\makeatother

\newtheorem{thm}{Theorem}[section]

\newtheorem{lem}[thm]{Lemma}
\newtheorem{prop}[thm]{Proposition}

\theoremstyle{definition}
\newtheorem{defn}[thm]{Definition}
\newtheorem{rem}[thm]{Remark}
\newtheorem{exa}[thm]{Example}

\numberwithin{equation}{section}

\begin{document}


\baselineskip=17pt



\title{Metric characterisation of connectedness for topological spaces}

\author[I. Weiss]{Ittay Weiss}
\address{School of Computing, Information, and Mathematical Sciences \\ The University of the South Pacific\\
Suva, Fiji}
\email{weittay@gmail.com}

\date{}

\begin{abstract}
Connectedness, path connectedness, and uniform connectedness are well-known concepts. In the traditional presentation of these concepts there is a substantial difference
between connectedness and the other two notions, namely connectedness
is defined as the absence of disconnectedness, while path connectedness
and uniform connectedness are defined in terms of connecting paths
and connecting chains, respectively. In compact metric spaces uniform
connectedness and connectedness are well-known to coincide, thus the
apparent conceptual difference between the two notions disappears. Connectedness in topological spaces can also be defined in terms of chains governed by open coverings in a manner that is more reminiscent of path connectedness.
We present a unifying metric formalism for connectedness, which encompasses both connectedness of topological spaces and uniform connectedness of uniform spaces, and which further extends to a hierarchy of notions of connectedness.
\end{abstract}



\maketitle

\section{Introduction}

The notions of connectedness and path connectedness for topological
spaces are, of course, well-known and, while the concepts are related,
their standard formulations are almost diametrically opposite each other, in
the following sense. Path connectedness is a positive condition in
that if a space is path connected, then for all pairs of points in
it there exists a path between the two points. Connectedness is a
negative condition in that if a space is connected, then no non-trivial
clopen sets exist. An equivalent definition of connectedness in a non-empty topological space $X$ is the following \v{C}ech-type formalism. Say that the points $x$ and $z$ are \emph{connected} if for every open covering of $X$ there exists a \emph{chain} from $x$ to $z$, namely a finite sequence $y_1,\dots, y_n$, with $y_1=x$ and $y_n=z$, such that for all $1\le k < n$ both $y_k$ and $y_{k+1}$ lie in the same open set of the covering. Then a space $X\ne \emptyset$ is connected if $x$ and $z$ are connected for all $x,z\in X$ (the same criterion can be given in a point-free fashion as well).  

If $X$ is a metric space, then open coverings of $X$ are related (roughly bijectively) to 
functions $R\colon X\to (0,\infty )$ (every such function yields the covering $\{B_{R(x)}(x)\mid x\in X\}$, 
and any open covering $\{U_\alpha \}$ yields a 
function $R$ by choosing, for each $x\in X$, a radius $R(x)>0$ such that $B_{R(x)}(x)$ is 
contained in some member of the covering). The above description of connectedness can then be 
given in terms of distances as follows. Say that two points $x$ and $z$ are \emph{connected} if for 
every such function $R$ there exists a \emph{walk} from $x$ to $z$, i.e., a finite 
sequence $y_1,\dots, y_n$, with $y_1=x$ and $y_n=z$, and such that, for all $1\le k < n$, 
either $y_k \in B_{R(y_{k+1})}(y_{k+1})$ or $y_{k+1}\in B_{R(y_k)}(y_k)$. 
Then a space $X\ne \emptyset$ is connected if $x$ and $z$ are connected for all $x,z\in X$. 

Connectedness for topological spaces can be defined in more than one way. It is well-known of course that for a topological space $X$ the conditions

\begin{itemize}
\item the only clopen subsets of $X$ are $\emptyset $ and $X$
\item every continuous function $X\to {\because}$ (where $\because$ is a discrete space) is constant
\end{itemize}
are equivalent. The proof is very nearly a tautology, yet the perspective each condition offers is slightly different, and each has its merits. Neither though can be said to be a positive condition in the same sense that a path connected space implies the existence of a connecting entity between any two points. This remark becomes clearer when one compares this situation with the definition of uniform connectedness (also known as Cantor connectedness) for uniform spaces. For a uniform space $X$ the following three conditions are equivalent (see \cite{PervinOnUnifConn}), and each may be taken as the definition of uniform connectedness for $X\ne \emptyset$

\begin{itemize}
\item the only uniformly clopen subsets of $X$ are $\emptyset $ and $X$
\item every uniformly continuous function $X\to {\because}$ is constant
\item for every entourage $E$ and all points $x,z\in X$ there exists a finite sequence $y_1,\ldots, y_n$ of points in $X$ such that $y_1=x$, $y_n=z$, and $y_kEy_{k+1}$ holds for all $1\le k < n$. 
\end{itemize}

Here a subset $U\subseteq X$ is uniformly open if there exists an entourage $E$ such that $E[x]\subseteq U$ for all $x\in U$, and $E[x]$ is the set $\{y\in X\mid xEy\}$. The proof of the equivalency is not hard at all, and, again, each condition has its merits. Clearly though, the third condition is a positive one, asserting the existence of a connecting entity for points, while the former two are not. 

In the presence of a metric the third condition in the list above is directly translated to the following condition which is often called \emph{chain connectedness}. A metric space $X$ is said to
be \emph{chain connected} if for all $\varepsilon>0$ and for
all $x,z\in X$ there exists an \emph{$\varepsilon$-walk} from $x$
to $z$, that is a sequence $y_{1},\dots,y_{n}\in X$ such that $y_{1}=x$,
$y_{n}=z$, and $d(y_{k+1},y_{k})\le\varepsilon$, for all $1\le k<n$. It is then quite easy to see that given a uniform space and a metric inducing it, the concepts of uniform connectedness and chain connectedness coincide.

Below, the metric observations above are examined and generalised in two directions. Firstly, by allowing $\varepsilon $ to vary at each point of the space one obtains a condition on a metric space equivalent to connectedness of the induced topological space. Second, by considering continuity spaces, one obtains a metric characterisation of connectedness for all topological spaces. The passage to continuity spaces involves replacing the non-negative real numbers by a value quantale, as well as abandoning the requirement that $d(x,y)=d(y,x)$, obviously affecting the proofs. 

There results a formalism rendering the concepts of uniform connectedness and connectedness as being completely analogous, and in particular providing the missing positive condition for topological connectedness analogous to the third condition listed above for uniform connectedness. The usefulness of the new formalism is also examined by looking at some well-known results from this new point of view. Moreover, the formalism suggests a natural hierarchy of notions of connectedness. 
\begin{rem}
Notions of connectedness were investigated early on in the development of topology,
e.g., for topological, uniform, syntopogenous, and bitopological spaces in, respectively,
\cite{Wiegandt, PervinOnUnifConn,ConnSyntopSp,LH1}, and surveyed in general
in \cite{LH3,LH2}. The current article can be seen as a companion of the former group of articles. 

To further place this work in the context of the study of connectedness we mention the categorical approach developed by Preuss in \cite{p1, p2, p3} dating back to the 1970's (see also \cite{Preuss}) and in particular \cite{ClosureOperators}, as well as the work of Herrlich (\cite{Herrlich2, Herrlich1}). The most natural definition of connectedness suitable for a notion of connectedness in an arbitrary category is that a non-empty space is connected when every mapping to a discrete space is constant. Replacing the class of discrete spaces in $\bf Top$ by a suitable class of objects in an arbitrary category yields a general theory of connectedness. Each of the notions of connectedness introduced below is characterised by a similar constancy condition with respect to discrete spaces in a suitably defined category, and thus the categorical approach subsumes ours. However, it is the positive aspect of our definition that sets it apart from the inherently negative categorical notion of connectedness. We refer the reader to \cite{ClosureOperators} for a quick review and a fuller reference list. 

We also mention Lowen's approach spaces (\cite{lowen1997approach}), a concept developed in order to unify topological and metric notions. Our use of Flagg's metric formalism via result given in \cite{WeissTopMetEquiv} shows that for the purposes of connectedness there is not need to leave the category of topological spaces, at least not up to an equivalence of categories. 
\end{rem}

\begin{rem}
Flagg's \cite{FlaggQuantales} investigates continuity spaces in a
broad context. However, it appears that most further work utilising
continuity spaces belongs largely to domain theory. Recent
work, namely \cite{WeissTopMetEquiv,BrunoWeiss,ChandWeiss} investigates
the applicability of continuity spaces in the context of topology.
This article is another result in that direction. 
\end{rem}
The plan of the paper is as follows. Section~\ref{sec:Continuity-spaces} briefly introduces continuity
spaces, which we call $V$-valued metric spaces, and recounts the
portion of their theory used in this work. Section~\ref{sec:Metric-connectivity} is where the
metric connectedness formalism is introduced and investigated. Section~\ref{heir} develops a single mechanism which reproduces connectedness and uniform connectedness as special cases, and extends further to an entire hierarchy of notions of connectedness. Finally, Section~\ref{consequences} explores some of the consequences of the metric approach to classical results, shedding some new light upon them.  

\section{Continuity spaces\label{sec:Continuity-spaces}}

For the results we present we can employ either Flagg's formalism through value quantales (\cite{FlaggQuantales}), or Koppermans formalism of value semigroups (\cite{KoppermanAllTop}). We shall briefly present Flagg's approach (it is shown in \cite{WeissComparing} that the two approaches are equally powerful as frameworks for topology). Continuity spaces, as introduced by Flagg (\cite{FlaggQuantales}), are a generalization
of metric spaces where the codomain of the metric function is a value
quantale, rather than the particular value quantale $[0,\infty]$
of non-negative reals. A \emph{value quantale} is a complete lattice
$V$, with the bottom and top elements satisfying $0<\infty $, and meet denoted by
$\wedge$, together with an associative and commutative binary operation
$+$ such that the conditions
\begin{itemize}
\item $a+\bigwedge S=\bigwedge(a+S)$, for all $a\in V$ and $S\subseteq V$
(where $a+S=\{a+s\mid s\in S\}$)
\item $a+0=a$ for all $a\in V$
\item $a=\bigwedge\{b\in V\mid b\succ a\}$
\item $a\wedge b\succ0$ for all $a,b\in V$ with $a,b\succ0$
\end{itemize}
hold. Here the meaning of $a\prec b$, or $b\succ a$, is that $b$
is \emph{well above} $a$, i.e., that for all $S\subseteq V$, if
$a\ge\bigwedge S$, then there exists $s\in S$ with $b\ge s$. It
is straightforward that 
\begin{itemize}
\item if $a\prec b$, then $a\le b$
\item if $a\le b\prec c$ or $a\prec b\le c$, then $a\prec c$
\item $0=\bigwedge\{\varepsilon\in V\mid\varepsilon\succ0\}$
\item if $a>0$, then there exists $\varepsilon\succ0$ with $a\nleq\varepsilon$. 
\end{itemize}
For the proofs of the following two properties in a value quantale
refer, respectively, to \cite[Theorem 1.6, Theorem 2.9]{FlaggQuantales}. 
\begin{itemize}
\item If $a\prec c$, then there exists $b\in V$ with $a\prec b\prec c$.
\item If $\varepsilon\succ0$, then there exists $\delta\succ0$ with $2\delta\le\varepsilon$
(where $2\delta=\delta+\delta$).
\end{itemize}
A \emph{continuity space} or a \emph{$V$-valued metric space} is
a triple $(V,X,d)$ where $V$ is a value quantale, $X$ is a set,
and $d\colon X\times X\to V$ is a function satisfying $d(x,x)=0$
and $d(x,z)\le d(x,y)+d(y,z)$, for all $x,y,z\in X$. With every
$V$-valued metric space $X$ one may associate the \emph{open ball
topology}, the one generated by the \emph{open balls} $B_{\varepsilon}(x)=\{y\in X\mid d(x,y)\prec\varepsilon\}$,
where $x\in X$ and $\varepsilon\succ0$. Every topological space
is metrizable in the sense that if $(X,\tau)$ is a topological space,
then there exists a value quantale $V$ and a $V$-valued metric structure
on $X$, such that the induced open ball topology is precisely $\tau$
(\cite[Theorem 4.15]{FlaggQuantales}). 

In \cite{WeissTopMetEquiv} it is shown that this construction extends
functorially to an equivalence ${\bf Met_c}\to{\bf Top}$, as follows.
Given value quantales $V$ and $W$, and metric spaces $(V,X,d)$ and
$(W,Y,d)$, declare a function $f\colon X\to Y$ to be \emph{continuous} if for all $x\in X$ and $\varepsilon\succ0_{W}$, there exists $\delta\succ0_{V}$
such that $d(fx,fy)\le\varepsilon$ for all $y\in X$ with $d(x,y)\le\delta$.
${\bf Met_c}$ is then the category whose objects are $(V,X,d)$ where
$V$ ranges over all value quantales and $X$ is a $V$-valued metric
space. The morphisms in ${\bf Met_c}$ are the continuous functions.
Flagg's construction of the open ball topology extends functorially
to form $\mathcal{O}\colon {\bf Met_c}\to{\bf Top}$, with $\mathcal{O}(V,X,d)$
sent to $X$ with the open ball topology, and $\mathcal{O}(f)=f$
for all continuous functions $f$. It then holds that $\mathcal{O}$
is an equivalence of categories. In light of this equivalence, $V$-valued
metric spaces with their continuous functions can be taken as models
for topology. 

The following useful property is \cite[Theorem 4.6]{FlaggQuantales}.
Given $x\in X$ and $S\subseteq X$ in a $V$-valued metric space
$X$, let $d(x,S)=\bigwedge_{s\in S}d(x,s)$. With respect to the
open ball topology, a set $C\subseteq X$ is closed if, and only if,
$d(x,C)=0$ implies $x\in C$, for all $x\in X$. In particular, a point $x\in X$ belongs to the closure of $S\subseteq X$ if, and only if, $d(x,S)=0$. 

It is shown in \cite{WeissOnMetrizability} that the situation for uniform spaces is similar, in the following sense. With any $V$-valued metric space $(V,X,d)$ one may associate the quasi uniform space whose entourages are generated by the sets $E_\varepsilon = \{(x,y)\in X\times X \mid d(x,y)\prec \varepsilon \}$, with $\varepsilon \succ 0$. Obviously, if $d$ is symmetric, then one obtains a uniform space. Every  uniform space $(X,{\mathcal E})$ is metrizable in the sense that there exists a value quantale $V$ and a symmetric metric space $(V,X,d)$ whose associated uniformity is $\mathcal E$. 

Now, declare a function $f\colon X\to Y$ to be \emph{uniformly continuous} if for all $\varepsilon\succ0_{W}$ there exists $\delta\succ0_{V}$
such that $d(fx,fy)\le\varepsilon$ for all $x,y\in X$ with $d(x,y)\le\delta$. Let ${\bf Met_u}$ be the subcategory of ${\bf Met_c}$ consisting of all $V$-valued metric spaces but only the uniformly continuous functions. Further, let ${\bf sMet_u}$ be the full subcategory of ${\bf Met_u}$ spanned by the symmetric $V$-valued metric spaces. The metrization above is part of an equivalence of categories ${\bf sMet_u}\simeq {\bf Unif}$. 

We may thus take the following approach. Every $V$-valued metric space is
silently endowed with the open ball topology and the system of entourages as above (yielding a quasi-uniformity which is a uniformity if $d$ satisfies symmetry). Every topological
space is automatically assumed to come with a $V$-valued metric structure
inducing its topology, and every uniform space is assumed to be equipped with a symmetric $V$-valued metric structure inducing its entourages. If the context requires it, when speaking of uniformities, it will be understood that a $V$-valued metric space is symmetric. 

\begin{rem}
As shown in \cite{WeissOnMetrizability}, one can obtain the metrizability of all quasi-uniform spaces by adaptating the definition of value quantale, namely by dropping the assumption that $+$ is commutative, and adjoining to every axiom its twin with respect to addition. For instance, the requirement that $a+0=a$ is augmented with the requirement that $0+a=a$, and similarly for the other axioms involving $+$. All of the above remains true for these slightly more general structures. It is the case that every quasi-uniform space $(X,\mathcal E)$ is metrizable in the above sense as well and there is then an equivalence of categories ${\bf Met_u}\simeq {\bf QUnif}$. It should be noted that the proofs below do not make use of the commutativity of $+$ in $L$, and thus the results can be phrased for quasi-uniform spaces. We choose to frame the results in Flagg's original commutative variant of value quantales for two reasons; simplicity and that currently the author does not know if it is necessary to pass to non-commutative value quantales in order to obtain the metrizability of quasi-uniform spaces, or whether every quasi-uniformity is obtained as the quasi-uniformity induced by a $V$-valued metric for a commutative $V$. If the latter is the case, then all the results below hold verbatim for quasi-uniform spaces. 
\end{rem}

We close this section with the following convenient observation,  namely that the metrizability of each topological space individually can be strengthened to give a mutual metrizability theorem for a family of topological spaces. 

\begin{thm}\label{mutual}
Let $\{(X_i,\tau_i)\}_{i\in I}$ be a collection of topological spaces indexed by a set $I$. There exists a single value quantale $V$ and, for each $i\in I$, a metric structure $d_i\colon X_i\times X_i \to V$ whose induced open ball topology is $\tau_i$. 
\end{thm} 

\begin{proof}
Let us quickly recall Flagg's metrization of a topological space $(X,\tau)$. For any set $S$ let $\Omega(S)$ be the collection of all down closed families of finite subsets of $S$, ordered by reverse inclusion, and with intersection as addition, giving rise to a value quantale. For any $T\subseteq S$, the collection ${\downarrow}T$ of all finite subsets of $T$ always belongs to $\Omega (S)$. Taking $S=\tau$ allows one to define $d\colon X\times X\to \Omega (\tau)$ by $d(x,y)={\downarrow}\tau_{x\rightarrow y}$, where $\tau_{x\rightarrow y}=\{U\in \tau \mid x\in U \implies y\in U\}$, resulting in an $\Omega(\tau)$-valued metric space. It is then an exercise in deciphering the definitions to show that the induced open ball topology is the original topology $\tau$. For the set indexed collection $\{(X_i,\tau _i)\}$, let $S=\bigcup _{i\in I}\tau_i$. The adaptation to the above required in order to obtain a mutual metrization of the entire family valued in $\Omega(S)$ is minor. 
\end{proof}

\section{Metric connectedness}\label{sec:Metric-connectivity}

Given a $V$-valued metric space $X$, consider a function $R\colon X\to V$
with the only restriction being that $R(x)\succ0$ for all $x\in X$
(which we abbreviate to $R\succ 0$). Such a function $R$ will be called a \emph{scale}. An \emph{$R$-step} (or simply
a \emph{step}, if $R$ is understood) is an ordered pair $(x,y)\in X\times X$
such that at least one of the conditions
\begin{itemize}
\item $d(x,y)\prec R(x)$ 
\item $d(y,x)\prec R(y)$ 
\end{itemize}
holds. An \emph{$R$-walk} (or simply a \emph{walk}) is a finite sequence
$x_{1},\dots,x_{n}$ of points in $X$ such that $(x_{k},x_{k+1})$
is a step for all $1\le k<n$. Such a walk is said to \emph{$R$-connect}
(or simply to \emph{connect}) $x_{1}$ and $x_{n}$, which are then
said to be \emph{$R$-connected}, and we write $x_{1}\sim_{R}x_{n}$,
clearly an equivalence relation. $X$ is said to be \emph{$R$-connected} if any two of its points are $R$-connected. We say that two points $x,y\in X$
are \emph{connected} if $x\sim_{R}y$ for all scales $R\colon X\to V$,
and we then write $x\sim y$, again an equivalence relation (indeed,
${\sim}=\bigcap_{R\succ0}\sim_{R}$). Given $\varepsilon\succ 0$, one
may consider the constant function $R_{\varepsilon}:X\to V$, $x\mapsto\varepsilon$, which we call a \emph{uniform scale}.
We say that $X$ is \emph{$\varepsilon$-connected} if it is $R_{\varepsilon}$-connected. Two points $x,y\in X$ are \emph{uniformly connected} if $x\sim_{R} y$ for all uniform scales $R$. For any scale $R$ and $z\in X$, let $C_{z}^{R}=\{x\in X\mid z\sim_{R}x\}$,
the \emph{$R$-connected component} of $z$. Obviously, $z\in C_{z}^{R}$. 

\begin{defn}
A non-empty $V$-valued metric space $X$ is \emph{connected} (resp. \emph{uniformly connected}) if  $x$ and $y$ are connected (resp. uniformly connected) for all $x,y\in X$. 
\end{defn}
 
\begin{rem}\label{alterStep}
\label{rem:alternative}If a step is instead defined to be a pair
$(x,y)$ such that at least one of
\begin{itemize}
\item $d(x,y)\le R(x)$
\item $d(y,x)\le R(y)$
\end{itemize}
holds, with the rest of the concepts above unchanged, then the final
notion of connectedness is the same as above. To see that,
note first that $a\prec b$ implies $a\le b$, and thus any walk under
the preceding definition is also a walk under the new definition. For
the converse, note that for all $\varepsilon\succ0$ there is a $\delta$
with $0\prec\delta\prec\varepsilon$, and then if $a\le\delta$, then
$a\prec\varepsilon$. Given a scale $R\colon X\to V$ there is then a scale $R'\colon X\to V$
with $0\prec R'(x)\prec R(x)$, for all $x\in X$. It then follows
that any $R'$-walk under the new definition is an $R$-walk under the old
one. We will choose whichever alternative to use based on convenience. 
\end{rem}

Given a scale $R$ on a $V$-valued metric space $X$ and a subset $S\subseteq X$ we write $B_R(S)=\{y\in X\mid \exists s\in S\colon d(s,y)\prec R(s)\}$, and we abbreviate $B_R (\{x\})$ to $B_R(x)$ for all $x\in X$. Notice that a subset $U\subseteq X$ is open in the induced topology of $(V,X,d)$ precisely when there exists a scale $R$ such that $B_R(U)=U$. Consequently, we declare a subset $U\subseteq X$ to be \emph{uniformly open} when there exists a uniform scale $R$ with $B_R(U)=U$. $U$ is \emph{uniformly clopen} if both $U$ and $X\setminus U$ are uniformly open. 

\begin{rem}
If $d$ is symmetric, then any uniformly open set is automatically uniformly clopen. 
\end{rem}

\begin{rem}
It is easy to see that if $(X,\mathcal E)$ is a uniform space and $(V,X,d)$ is a symmetric $V$-valued metric space inducing the uniformity $\mathcal E$, then a subset $U\subseteq X$ is uniformly open in the metric sense (namely, $U=B_R(U)$ for some uniform scale $R$) if, and only if, it is uniformly open in the uniform sense (namely, there exists an entourage $E$ such that $E[x]\subseteq U$ for all $x\in U$). 
\end{rem}

\begin{prop}\label{AreClopen}
In a $V$-valued metric space $X$, given any scale (resp. uniform scale) $R$
and $z\in X$, the $R$-connected component $C_z^R$ is clopen (resp. uniformly clopen). 
\end{prop}
\begin{proof}
Using the first definition of step, the first condition implies that
if $x\in C_{z}^{R}$, then $B_{R(x)}(x)\subseteq C_{z}^{R}$, and
thus $C_{z}^{R}$ is open (resp. uniformly open). By the second step condition, if $y\in X\setminus C_{z}^{R}$,
then $B_{R(y)}(y)\subseteq X\setminus C_{z}^{R}$, and thus $X\setminus C_{z}^{R}$
is open (resp. uniformly open).
\end{proof}

The next result is a tautology.

\begin{prop}\label{cont}
Let $(V,X,d)$ and $(W,Y,d)$ be metric spaces, valued, respectively, in $V$ and in $W$. A function $f\colon X\to Y$ between the underlying sets is continuous (resp. uniformly continuous) if, and only if, for every scale (resp. uniform scale) $R$ on $Y$ there exists a scale (resp. uniform scale) $S$ on $X$ such that $d(fx,fx')\prec R(fx)$ for all $x,x'\in X$ with $d(x,x')\prec S(x)$. 
\end{prop}

The following result shows that a continuous (resp. uniformly continuous) image of a connected (resp. uniformly connected) $V$-valued metric space is connected (resp. uniformly connected). 

\begin{prop}\label{ImageIsConn}
Let $f\colon (V,X,d)\to (W,Y,d)$ be a continuous (resp. uniformly continuous) function between metric spaces. If two points $x, x'\in X$ are connected (resp. uniformly connected), then so are $f(x)$ and $f(x')$. 
\end{prop}

\begin{proof}
Let $R$ be a scale (resp. uniform scale) on $Y$. By Proposition~\ref{cont} there exists a scale (resp. uniform scale) $S$ on $X$ with $d(fx,fx')\prec R(fx)$ for all $x,x'\in X$ with $d(x,x')\prec S(x)$. We may now choose an $S$-walk in $X$ from $x$ to $x'$ and observe that this latter condition implies at once that the image of that walk is an $R$-walk in $Y$ from $f(x)$ to $f(x')$, which completes the proof. 
\end{proof}

We now establish the equivalency between the classical notion of connectedness of a topological space (i.e., a non-empty one having no non-trivial clopen sets) and the metric notion of connectedness, as well as the equivalency between the classical notion of connectedness of a uniform space (i.e., a non-empty one having no non-trivial uniformly clopen sets) and the notion of metric notion of uniform connectedness. 

\begin{thm}
Let $(V,X,d)$ be a non-empty $V$-valued metric space, $\mathcal{O}$
its associated open ball topology, and $\mathcal E$ its associated uniformity (if $d$ is symmetric). Then $(V,X,d)$ is connected (resp. uniformly connected) if,
and only if, $(X,\mathcal{O})$ is connected (resp. $(X,\mathcal E)$ is uniformly connected). 
\end{thm}
\begin{proof}
Assume that $(X,\mathcal{O})$ is connected (resp. $(X,\mathcal E)$ is uniformly connected) and let $R$ be a scale (resp. uniform scale) on $X$. For an arbitrary $x\in X$ the $R$-connected
component $C_{x}^{R}$ is clopen (resp. uniformly clopen) and thus $C_{x}^{R}=X$, completing the proof.

Assume now that $(X,\mathcal{O})$ is not connected (resp. $(X,\mathcal E)$ is not uniformly connected) and let $C\subseteq X$
be a non-trivial clopen (resp. uniformly clopen) subset. There exists then a scale (resp. uniform scale) $R$ on $X$ such that $B_R(C)=C$ and $B_R(\hat C)=\hat C$, where $\hat C = X\setminus C$. We claim that this scale demonstrates a separation between any point in $C$ and any point in $\hat C$. Indeed, if a walk existed connecting $C$ to $\hat C$, then it must posses a single step $(x,y)$ with $x\in C$ and $y\notin C$. However, $d(x,y)\prec R(x)$ contradicts $B_R(x)\subseteq C$ while $d(y,x)\prec R(y)$ contradicts $B_R(\hat C)\subset \hat C$. 
\end{proof}
\begin{defn}
Let $X$ be a $V$-valued metric space and let $z\in X$. The \emph{connected
component} of $x$ is the set $C_{z}=\{x\in X\mid z\sim x\}$. 
\end{defn}
Noting that $C_{z}=\bigcap_{R\succ0}C_{z}^{R}$, and seeing that $C_{z}^{R}$
is clopen, it follows immediately that $C_{z}$ is closed (but not
necessarily open). 
\begin{rem}
It is interesting to compare and contrast the notions of $R$-connectedness and path-connectedness, and to note that some of the qualitative differences
between $R$-connectedness and connectedness (e.g., that  $R$-connected
components are clopen while connected components are closed) are attributed to the extra quantification over $R\succ 0$
required for connectedness. Numerous (and perhaps all) of the fundamental
facts about connectedness can be proven using the metric formalism
presented above. Sometimes the proof is shorter, other times it may
be a bit longer, but always the proof employs arguments strongly reminiscent
of path connectedness arguments, and the proof is, arguably, closer to one's natural
intuition of connectedness. Examples and further discussion are given below. 
\end{rem}

As stated in the introduction, for ordinary metric spaces, i.e., for
symmetric $[0,\infty]$-valued metric spaces, connectedness implies
uniform connectedness, and under compactness the two notions coincide.
We now establish the analogous result for general $V$-valued metric
spaces, and thus for all topological spaces. 

\begin{thm}\label{unifConnPlusCompIsConn}
Let $X$ be a non-empty $V$-valued metric space. If $X$ is connected, then it is uniformly connected, and the converse holds if $X$ is compact.
\end{thm}
\begin{proof}
If $X$ is connected, then it is $R$-connected for all $R\succ 0$,
certainly then also for all $R_{\varepsilon}$, $\varepsilon\succ 0$.
Assume now that $X$ is compact and uniformly connected. Let $R\succ 0$
be given. For each $x\in X$ let $R'(x)\succ 0$ satisfying $2R'(x)\le R(x)$,
and consider the open covering $\{B_{R'(x)}(x)\mid x\in X\}$, from
which we may extract a finite subcovering, say with centres $x_{1},\dots,x_{n}$,
and let $\varepsilon=R'(x_{1})\wedge\dots\wedge R'(x_{k})$. Suppose
now that $d(x,y)\le\varepsilon$ (cf. Remark~\ref{rem:alternative}).
There is then $1\le k\le n$ with $x\in B_{R'(x_{k})}(x_{k})$. We
then have $d(x_{k},x)\prec R'(x_{k})$, and thus $d(x_{k},y)\le d(x_{k},x)+d(x,y)\le R'(x_{k})+R'(x_{k})\le R(x_{k})$.
It now follows that $(x,x_{k},y)$ is an $R$-walk. A similar observation
holds if $d(y,x)\le\varepsilon$. We thus showed that any $\varepsilon$-step
$(x,y)$ may be split into a $2$-step $R$-walk, and thus any $\varepsilon$-walk
can be refined to an $R$-walk without altering the end points. As
$X$ is assumed uniformly connected, it follows that it is connected. 
\end{proof}

Let us illustrate the different flavour that this result, and the formalism we are presenting, gives to classical results. This relationship between connectedness and uniform connectedness holds not just for metric spaces. The following theorem is well known, but the proof we provide shows that the result is truly just a reflection of the familiar metric situation, taking advantage of the fact that through the concept of continuity spaces the models for topology and for (quasi) uniform spaces are the same (i.e., the objects in the categories are the same, namely $V$-valued metric spaces, and only the morphisms change). Suppose that a metric space $(V,X,d)$ and a uniformity $\mathcal E$ on $X$ are compatible in the sense that $d$ induces $\mathcal E$. It is then easy to see that the topology induced by $\mathcal E$ coincides with the open ball topology induced by $d$. 

\begin{thm}
Let $(X,\mathcal E)$ be a uniform space and $(X,\mathcal O)$ its induced topology. If $(X,\mathcal O)$ is connected, then so is $(X,\mathcal E)$, and the converse holds if $(X,\mathcal O)$ is compact. 
\end{thm}

\begin{proof}
Let $(V,X,d)$ be a metric space inducing the uniformity $\mathcal E$. If $(X,\mathcal O)$ is connected, then so is $(V,X,d)$, which is thus uniformly connected, implying that $(X,\mathcal E)$ is uniformly connected. If $(X,\mathcal O)$ is compact, i.e., $(V,X,d)$ is compact, then if $(X,\mathcal E)$ is uniformly connected, then $(V,X,d)$ is uniformly connected, and thus connected, implying that $(X,\mathcal O)$ connected. 
\end{proof}

\section {Hierarchies of connectedness}\label{heir}

Inspection of the above results and proofs reveals that a portion of the above is completely determined just by the choice of which scale functions one allows. Some of the proofs then are a formal consequence of that choice. We distill the formal aspects in this section.

Let $\bf Met_{all}$ be the category whose objects $(V,X,d)$ are all $V$-valued metric spaces (where $V$ varies over all value quantales) and whose morphisms $f\colon (V,X,d)\to (W,Y,d)$ are all functions $f\colon X\to Y$ between the underlying sets, with composition and identities given as in $\bf Set$. A \emph{scale system} $\Sigma$ is a specification, for each object $X=(V,X,d)$, of a non-empty set $\Sigma_X$ of scales on $X$, called \emph{$\Sigma$-scales}. 

\begin{exa}
Examples of scale systems include the following. 
\begin{itemize}
\item The collection $\Sigma_a$ of all scales. 
\item The collection $\Sigma_0$ of all scales $R\colon X\to V$ such that there exists $\varepsilon \succ 0$ with $R(x)\ge \varepsilon$, for all $x\in X$. 
\item The collection $\Sigma_u$ of all uniform scales. 
\item Fix, in each value quantale $V$, an element $\varepsilon_V\succ 0$, and denote the entire selection by ${\bf g} $. The scale system $\Sigma_{\bf g}$ consists, for each metric space $(V,X,d)$, of all scales $R\colon X\to V$ with $R(x)\ge \varepsilon_V$, for all $x\in X$. 
\item For a given metric space $(V,X,d)$, a function $\alpha \colon V\to V$, and a scale $R$ on $X$ say that a point $x\in X$ is a \emph{point of reference} for $R$ relative to $\alpha$ if $R(y)\ge \alpha(d(x,y))$, for all $y\in X$. We then say that $R$ has \emph{expansion rate} greater than $\alpha$. Fixing a function $\alpha \colon V\to V$ for each value quantale $V$, and denoting the entire selection by $\bf e$, the scale system $\Sigma_{\bf e}$ consists, for each metric space $(V,X,d)$, of all scales $R$ on $X$ with expansion rate greater than $\alpha$. 
\end{itemize}
\end{exa}

Let us fix a scale system $\Sigma$.

\begin{defn}
A morphism $f\colon (V,X,d)\to (W,Y,d)$ 
in $\bf Met_{all}$ is \emph{$\Sigma$-continuous} 
if for every $\Sigma$-scale $R$ on $Y$ there exists a $\Sigma$-scale $S$ 
on X such that $d(fx,fx')\prec R(fx)$ for all $x,x'\in X$ with $d(x,x')\prec S(x)$.
\end{defn} 
The following observation is trivial. 
\begin{prop}
Restricting $\bf Met_{all}$ to the $\Sigma$-continuous functions yields a category, denoted by $\bf Met_\Sigma$.
\end{prop} 

$\Sigma_a$-continuity is precisely ordinary continuity, and both $\Sigma_0$-continuity and $\Sigma_u$-continuity give the notion of uniform continuity in the usual sense. $\Sigma_{\bf g}$-continuity is a weaker notion than ordinary continuity as it allows functions with small jump discontinuities (small is determined by the choice of $\bf g$) to be considered $\Sigma_{\bf g}$-continuous. In particular, generally speaking, $\Sigma$-continuity need not imply continuity. $\Sigma_{\bf e}$-continuous functions include all continuous functions as well as ones with large jump discontinuities as long as their global behaviour is tame relative to the choice of $\bf e$. 

\begin{defn}
Let $(V,X,d)$ be an object in $\bf Met_{all}$. A subset $A\subseteq X$ is \emph{$\Sigma$-open} if there exists a $\Sigma$-scale $R$ on $X$ such that $A=B_R(A)$. The set $A$ is \emph{$\Sigma$-closed} if $X\setminus A$ is $\Sigma$-open, and finally $A$ is \emph{$\Sigma$-clopen} if it is both $\Sigma$-open and $\Sigma$-closed.  
\end{defn}

\begin{rem}
It is easily seen that any $\Sigma$-continuous 
function $f\colon (V,X,d)\to (W,Y,d)$ has the property that the inverse 
image $f^{\leftarrow}(U)$ 
is $\Sigma$-open whenever $U\subseteq Y$ is $\Sigma$-open. 
The converse however need not hold. For instance, consider the scale system $\Sigma=\Sigma_u$ consisting of all uniform scales. Then the $\Sigma$-open sets are the uniformly open ones, and those are automatically uniformly clopen, provided $d$ is symmetric. As noted above, the $\Sigma$-continuous functions are the uniformly continuous ones. 
Let $([0,\infty ], \mathbb Q, d)$ be the usual metric structure on the rationals, $d(x,y)=|x-y|$. $\mathbb Q$ is then 
uniformly connected, so its clopen sets are trivial and thus are preserved under the inverse image of every $f\colon \mathbb Q \to \mathbb Q$. However, obviously, not all such functions are uniformly continuous. 
\end{rem}

The proof of the following result is trivial. 

\begin{prop}
Let $(V,X,d)$ be a $V$-valued metric space. Every $\Sigma$-open  (resp. $\Sigma$-closed, $\Sigma$-clopen) set is open (resp. closed, clopen) in the open ball topology. The sets $\emptyset$ and $X$ are always $\Sigma$-clopen.
\end{prop}

$\emptyset $ and $X$ are called the \emph{trivial} $\Sigma$-clopen subsets of $X$. 

We now turn to relate these notions to $\Sigma$-connectedness. We note that Remark~\ref{alterStep} does not hold in the context of an arbitrary scale system $\Sigma$. We thus always use the definition of step utilising $\succ $ (this detail crops up in several places below). 

\begin{prop}\label{isClopen}
For any metric space $(V,X,d)$, a $\Sigma$-scale $R$, and $x\in X$, the connected component $C_x^R$ is $\Sigma$-clopen. 
\end{prop}

\begin{proof}
Let $t\in X$. If $t\in C_x^R$, then there exists a walk from $x$ to $t$. By the first condition in the definition of step, extending this walk by one more step allows one to get to any point in $B_R(x)$, and thus $B_R(t)\subseteq C_x^R$, showing that $B_R(C_x^R)=C_x^R$, which is thus $\Sigma$-open. If $t\notin C_x^R$, then no point in $B_R(t)$ can be connected to $x$, since by the second condition in the definition of step any such walk can be extended by one more step to a walk from $x$ to $t$. Thus $B_R(X\setminus C_x^R)=X\setminus C_x^R$, which is thus $\Sigma$-open.  
\end{proof}

\begin{defn}
A non-empty $V$-valued metric space $(V,X,d)$ is \emph{$\Sigma$-connected} if $X$ is $R$-connected for all $\Sigma$-scales $R$ on $(V,X,d)$. The \emph {$\Sigma$-connected component} of $x\in X$ is the set $C_x^\Sigma = \bigcap C_x^R$, where $R$ ranges over all $\Sigma$-scales on $(V,X,d)$. 
\end{defn}

A direct consequence of Proposition~\ref{isClopen} is the following result. 
\begin{lem}\label{isClosed}
For any metric space $(V,X,d)$ and $x\in X$, the connected component $C_x^\Sigma$ is closed in the open ball topology.
\end{lem}

\begin{thm}
Let $f\colon (V,X,d)\to (W,Y,d)$ be a $\Sigma$-continuous function. If $X$ is $\Sigma$-connected, then so is the image $f(X)$. 
\end{thm}

\begin{proof}
Given a $\Sigma$-scale $R$ on $(W,Y,d)$ and two points $f(x),f(z)\in Y$ we must construct an $R$-walk from $f(x)$ to $f(z)$. Since $f$ is $\Sigma$-continuous there is a suitable $\Sigma$-scale $S$ on $(V,X,d)$ for which $f$ transforms $S$-walks to $R$-walks. Since an $S$-walk from $x$ to $z$ exists, the proof is complete.
\end{proof}

\begin{defn}
A space $(V,X,d)$ is \emph{$\Sigma$-discrete} if every $x\in X$ is \emph{$\Sigma$-isolated} in the sense that $\{x\}$ is $\Sigma$-clopen. The symbol $\because $ will be used for an arbitrary $\Sigma$-discrete space. 
\end{defn}

\begin{prop}
Every non-empty $\Sigma$-discrete space $\because$ is \emph{totally $\Sigma$-disconnected} in the sense that $C_x^\Sigma = \{x\}$ for all $x\in X$. 
\end{prop}

\begin{proof}
Let $x\in X$ and $R$ be a $\Sigma$-scale witnessing the isolation of $x$. Then for all $y\in X$ with $x\ne y$, no $R$-walk between $x$ and $y$ exists since any step connecting $x$ to any $t$ in $\hat x=X\setminus \{x\}$ must satisfy either $d(x,t)\prec R(x)$ or $d(t,x)\prec R(t)$. But the former implies that $t\in B_R(x)=\{x\}$, while the latter implies that $x\in B_R(\hat x)=\hat x$, but neither is possible. 
\end{proof}

Let us consider the set $\{\bullet, \circ\} $ as a metric space valued in $[0,\infty ]$, with $d(\bullet, \circ)=d(\circ, \bullet)=\infty $.

\begin{lem}
For a non-empty metric space $(V,X,d)$ the following conditions are equivalent
\begin{itemize}
\item $X$ is $\Sigma$-connected.
\item Every $\Sigma$-continuous function $X\to {\because}$ to a $\Sigma$-discrete space is constant.
\item Every $\Sigma$-continuous function $X\to \{\bullet, \circ\}$ is constant. 
\item $X$ does not have any non-trivial $\Sigma$-clopen sets. 
\end{itemize}
\end{lem}

\begin{proof}
Suppose that $(V,X,d)$ is $\Sigma$-connected. A $\Sigma$-continuous function $f\colon X\to {\because}$ must have a $\Sigma$-connected image, which thus must lie in a $\Sigma$-connected component of $\because$, which is a singleton, and thus $f$ is constant. 

Suppose now that every $\Sigma$-continuous function $f\colon X\to {\because} $ is constant. To show that any $\Sigma$-continuous function $f\to \{\bullet, \circ \}$ is constant, it suffices to show that the codomain is $\Sigma$-discrete. Indeed, for any $\Sigma$-scale $R$ on $\{\bullet, \circ\}$ (and at least one such scale exists) $B_R(\bullet)=\{\bullet\}$ and $B_R(\circ)=\{\circ\}$ are easily computed. 

Assume now that every $\Sigma$-continuous function $X\to \{\bullet, \circ\}$ is constant. If a non-trivial $\Sigma$-clopen set $C\subseteq X$ existed, then consider the indicator function $\xi_C\colon X \to \{\bullet ,\circ\}$, which is clearly not constant. To obtain a contradiction we show that $\xi_C$ is $\Sigma$-continuous. Indeed, let $R$ be a $\Sigma$-scale witnessing the fact that $C$ is $\Sigma$-clopen. With that scale the condition $d(x,y)\prec R(x)$ implies that either $x,y\in C$ or that $x,y\notin C$, and thus that $\xi_C(x)=\xi_C(y)$, so no matter which $\Sigma$-scale is chosen on $\{\bullet,\circ\}$, a suitable $\Sigma$-scale on $X$ (namely $R$) exists. 

Finally, suppose that $X$ does not have any non-trivial $\Sigma$-clopen sets. Taking any $x\in X$ the $\Sigma$-connected component $C_x^\Sigma$ is $\Sigma$-clopen, and non-empty, and thus $C_x^\Sigma=X$. 
\end{proof}

Every choice of scale system $\Sigma$ thus yields a notion of connectedness, and we obtain a rather large hierarchy of notions of connectedness. Different scales may lead to the same notion of connectedness and some notions of connectedness need not be topological invariants. Some general results about $\Sigma$-connectedness are given in the next section. For now we will just notice the following triviality.  

\begin{thm}
Let $\Sigma_1$ and $\Sigma_2$ be two scale systems. If $\Sigma_1\subseteq \Sigma_2$, then any space $(V,X,d)$ which is $\Sigma_2$-connected, is also $\Sigma_1$-connected. In particular, any space $(V,X,d)$ whose induced topological space is connected in the usual sense, is $\Sigma$-connected for all scale systems $\Sigma$. 
\end{thm}

Intuitively, $\Sigma$-connectedness is a concept that may consider spaces with various notions of gaps to still be connected, while the classical notion of topological connectedness is the strongest kind of connectedness in this hierarchy. For instance, $\Sigma_u$-connectedness and $\Sigma_0$-connectedness both correspond to uniform connectedness. $\Sigma_{\bf g}$-connectedness allows for spaces with small cracks (where $\bf g$ determined how narrow a crack must be) to still be considered connected. Finally, $\Sigma_{\bf e}$-connectedness allows for spaces with expanding cracks to be considered connected as long as the rate of expansion of the crack is small relative to the choice of $\bf e$.   

To close this section we note that path connectedness can not be captured as $\Sigma$-connectedness for any scale system $\Sigma$, for (at least) two reasons: 1) path connected components, unlike $\Sigma$-connected components, need not be closed; and 2) path connectedness is a strictly stronger notion than connectedness, but connectedness is the strongest property in the hierarchy.  

\section{Classical results revisited}\label{consequences}

In this section we revisit some fundamental results pertaining to connectedness, examined through the lens of the metric formalism developed above. In short: old results - new perspective. We also prove some classical results at the level of $\Sigma$-connectedness, thus identifying the results as formal consequences of the scale formalism above. 

Let $(V,X,d)$ be a $V$-valued metric space. Recall that for any set $S\subseteq X$, a point $x\in X$ belongs to the closure of $S$ if, and only if, $d(x,S)=0$. Denote by $\overline{S}$ the closure of $S$ under the open ball topology on $X$.

\begin{thm}
If $C$ is a $\Sigma$-connected subset of a space $(V,X,d)$, then any subset $C\subseteq D\subset \overline{C}$ is $\Sigma$-connected.
\end{thm}

\begin{proof}
To show $D$ is connected let $R$ be a scale and $x,y\in D$, namely $d(x,C)=0=d(y,C)$. Since $R(x)\succ 0$, it follows that there exists $x'\in C$ with $d(x,x')\le R(x)$. Similarly, there exists $y'\in C$ with $d(y,y')\le R(x)$. Since $C$ is connected, there exists a walk in $C$ from $x'$ to $y'$, which we may now augment at each end to obtain a walk in $D$ from $x$ to $y$, as required. 
\end{proof}

\begin{thm}
The closed interval $[a,b]$ with the Euclidean topology is connected. 
\end{thm}
\begin{proof}
We present two proofs, one which appeals to compactness, and one that does not. 

With knowledge of compactness of $[a,b]$: By Theorem~\ref{unifConnPlusCompIsConn}, it suffices to show that $[a,b]$ is uniformly connected. And indeed, given a uniform scale $R$ the existence of a walk between any two points of $[a,b]$ follows trivially by the fact that $\mathbb R$ is archimedean.

Without appealing to compactness: Given a scale $R$ consider $C=C_a^R$, the $R$-connected component of $a$. It suffices to show that $C=[a,b]$, so assume that is not the case and let $t$ be the infimum of $[a,b]\setminus C$. 

Let us write $\hat C = [a,b]\setminus C$. Remembering that for all $x\in [a,b]$ if $x\in C$, then $B_R(x)\subseteq C$, and if $x\in \hat C$, then $B_R(x)\subseteq \hat C$ we proceed as follows. If $t=1$, then $C=[a,b)$. 
But $b\in \hat C$, implying $B_R(b)\subseteq \hat C $, an impossibility. We thus conclude that $t<b$. If $t\in C$, then $B_R(t)\subseteq C$, which 
contradicts the definition of $t$. Knowing that $t\in \hat C$, it follows that $B_R(t)\subseteq \hat C$, which would lead to a contradiction unless $t=a$. But certainly $a\in C$ and we already established that $t\notin C$, a final blow. 
\end{proof}

\begin{thm}
If $X$ is a path connected space, then $X$ is connected. 
\end{thm}

\begin{proof}
Let $d$ be a $V$-valued metric inducing the topology on $X$, let $R$ be a scale on $X$, and $x,y\in X$ two points, for which we may find a connecting path $[0,1]\to X$, where we treat $[0,1]$ as a metric subspace of $\mathbb R$ with the Euclidean metric. Now repeat the argument of Proposition~\ref{ImageIsConn}. 
\end{proof}

A \emph{subspace} of a $V$-valued metric space $(V,X,d)$ is a subset $Y\subseteq X$ equipped with the metric $d'\colon Y\times Y\to V$ given by $d'(y_1,y_2)=d(y_1,y_2)$. A scale system $\Sigma$ is \emph{hereditary} if for every space $(V,X,d)$ and every $\Sigma $-scale $R$ on $X$, the restriction of $R$ to any subset $Y\subseteq X$ is a $\Sigma$-scale on the subspace $(V,Y,d)$. 

\begin{thm}
Let $\Sigma$ be a hereditary scale system, $(V,X,d)$ an ambient space, and $\{Y_k\}_{k\ge 1}$ a countable collection of $\Sigma$-connected subspaces of $X$. If $Y_k\cap Y_{k+1}\ne \emptyset$ for all $k\ge 1$, then the subspace $Y=\bigcup_k Y_k$ is $\Sigma$-connected. 
\end{thm}

\begin{proof}
Given a $\Sigma$-scale $R$ and two points $x,z\in Y$, we must construct a walk from $x$ to $z$. Without loss of generality, $x\in Y_{k_x}$ and $z\in Y_{k_z}$ with $k_x \le k_z$. We may now find a walk in $Y_{k_x}$ from $x$ to a point common with $Y_{k_x+1}$, concatenate it with a walk in $Y_{k_x+1}$ to a point common with $Y_{k_x+2}$, and so on, until we concatenate with walk in $Y_{k_z}$ ending at $z$. The entire concatenation is finite (as a finite union of finitely long walks), and thus obviously a walk residing in $Y$ and connects $x$ to $z$. 
\end{proof}

\begin{rem}
The familiar special case of this result for connectedness follows by noting that $\Sigma_a$ is hereditary and by applying mutual metrizablity (Theorem~\ref{mutual}).
\end{rem}

The following result has a similar proof, and should be followed by a similar remark. We omit both. 

\begin{thm}
Let $\Sigma$ be a hereditary scale system, $(V,X,d)$ an ambient space, and $\{Y_i\}_{i\in I}$ a family of $\Sigma$-connected subspaces. If $\bigcap _i Y_i\ne\emptyset$, then $\bigcup_i Y_i$ is $\Sigma$-connected. 
\end{thm}

We conclude this work with another well-known result, showing yet again the intuitive nature of the arguments resulting from the metric formalism. 

Let $\{(V_i,X_i,d)\}_{i\in I}$ be a family of spaces indexed by a set $I$. Let $X$ be the product of the underlying sets $X_i$ and $V$ the product of the underlying sets $V_i$, where we write, as usual, $x=(x_i)$ for a typical element in $X$, and similarly for $V$. Obviously, we may define $d\colon X\times X\to V$ by means of $d(x,y)=(d(x_i,y_i))$. Endowing $V$ with the obvious coordinate-wise structure it is a trivial matter to verify that $V$ satisfies all the axioms of a value quantale except for one; $a,b\succ 0$ need not imply $a\wedge b\succ 0$. Let us ignore this difficulty for a short while and pretend that $V$ is a value quantale and further that $a\in V$ satisfies $a\succ 0$ if, and only if, $a_i\succ 0$ for all $i\in I$ and $a_i=\infty $ for all but finitely many $i\in I$. The following arguments are thus incorrect, as they rely on this fantasy, but we shall show below how the metric formalism actually does support the arguments by means of a suitable construction. 

Thus, under the incorrect assumption that $V$ is a value quantale, it is also the case that $(V,X,d)$, with $d(x,y)=(d(x_i,y_i))$ is a $V$-valued metric space. Our aim is to show that if $(V_i,X_i,d)$ is $\Sigma$-connected for all $i\in I$, then so is $(V,X,d)$. Thus, let $R$ be a $\Gamma$-scale on $X$ and $x,y\in X$. We proceed informally in order to construct a walk from $x$ to $y$. Since $R(x)\succ 0$ it follows that $R(x)_i\succ 0$ at all but finitely many $i\in I$. Making the further assumption that each $(V_i,X_i,d)$ is \emph{locally finite}, i.e., $d(x_i,y_i)<\infty $ for all $x_i,y_i\in X_i$, it follows that in one step we can change some of the coordinates of $x$ to obtain $\hat {x}\in X$ agreeing with $y$ at all but finitely many coordinates. To change one more of the coordinates where $\hat {x}$ and $y$ still do not agree, say at position $i_0\in I$, we proceed as follows. Consider the function $X_{i_0}\to X$ obtained by replacing the $i_0$-th coordinate in $\hat x$ by $z$, for each $z\in X_i$. Composing with $R$ then yields a function $R_{i_0}\colon X_{i_0}\to L$. Let us say that $\Sigma$ is \emph{product compatible} if each such $R_{i_0}$ is a $\Sigma$-scale. Since $X_{i_0}$ is $\Sigma$-connected there exists a walk in $X_{i_0}$ connecting $x_{i_0}$ to $y_{i_0}$. From the definition of $d$ on $X$ it follows that that walk gives rise to a walk in $X$ simply by augmenting each step with all the remaining coordinates in $\hat x$. We are thus able to change each of the coordinates in $\hat x$ in turn to match the coordinates of $y$ and as only finitely many coordinates need to be taken care of, a finite walk from $\hat x$ to $y$ exists.

We now present the formal details. The technical difficulty with the failure of the component-wise product of value quantales forming a value quantale is solved by appealing to the construction $\Gamma$ from \cite{WeissComparing}: the component-wise structure on the product $V$ of value quantales $\{V_i\}_{i\in I}$,  together with the choice of positives to be all $a\in V$ with $a_i\succ 0$ for all $i\in I$ and $a_i=\infty $ for all but finitely many $i\in I$, forms a Kopperman value semigroup, and applying $\Gamma$ to it yields a value quantale $W$ together with an injection $V\to W$. The above arguments, with slight changes, are unaffected, so we may now state the general result. We do not repeat the details of the proof. 

\begin{thm}
Assume $\Sigma$ is a product compatible scale system. Let $\{(V_i,X_i,d)\}_{i\in I}$ be a set indexed family of spaces. Let $W$ be the product value quantale of $\{V_i\}_{i\in I}$ as constructed above, and $d\colon X\times X\to V$ the product metric on $X$, the product of the sets $\{X_i\}_{i\in I}$. If each $(V_i,X_i,d)$ is locally finite and $\Sigma$-connected, then $(V,X,d)$ is $\Sigma$-connected (and locally finite, but that is not the issue). 
\end{thm}

The familiar result that the product of a family of connected topological spaces $\{(X_i, \tau _i)\}_{i\in I}$ is connected follows by noting that $\Sigma_a$ is (trivially) product compatible, that 
the Flagg metrization of $X_i$ by means of $d\colon X_i\times X_i \to \Omega(\tau_i)$ is locally finite (since $\infty \in \Omega(\tau_i)$ is the empty collection), and that the 
metric product above agrees with the topological product (in the sense that the 
open ball topology induced by $d\colon X\times X\to W$ is the product of the open ball topologies induced by each $d\colon X_i\times X_i\to V_i$). 

\subsection*{Acknowledgements}
The author thanks David Milovich for helpful remarks and the referee for suggestions that led to the current form of the article.






\normalsize
\baselineskip=17pt


\bibliographystyle{plain}
\bibliography{refs}

\end{document}